\documentclass{amsart}
\title{Elmendorf constructions for $G$-categories and $G$-posets}

\author{Jonathan Rubin}
\address{University of California Los Angeles,
Los Angeles, CA 90095}
\email{jrubin@math.ucla.edu}


\date{\today}
\usepackage{amsmath,amssymb,amsthm,stackrel,tikz,enumitem,mathrsfs,physics}

\renewcommand{\b}[1]{\mathbf{#1}}
\newcommand{\bb}[1]{\mathbb{#1}}
\renewcommand{\t}[1]{\textnormal{#1}}
\newcommand{\s}[1]{\mathscr{#1}}
\renewcommand{\op}{\textnormal{op}}
\newcommand{\ve}{\varepsilon}

\newcommand{\ol}[1]{\overline{#1}}

\theoremstyle{plain}
\newtheorem{thm}{Theorem}[section]
\newtheorem{prop}[thm]{Proposition}
\newtheorem{lem}[thm]{Lemma}
\newtheorem{cor}[thm]{Corollary}

\newtheorem*{thmcomp}{Theorems \protect{\ref{thm:compElm} and \ref{thm:Elmpostop}}}
\newtheorem*{thmElm}{Theorems \protect{\ref{thm:catcoal} and \ref{thm:Elmpos}}}

\theoremstyle{definition}
\newtheorem{defn}[thm]{Definition}
\newtheorem{cex}[thm]{Counterexample}

\theoremstyle{remark}
\newtheorem{rem}[thm]{Remark}

\begin{document}
\maketitle

\begin{abstract} We introduce new Elmendorf constructions for equivariant categories and posets, and we prove that they are compatible with the classical topological one. Our constructions are more concrete than their model-categorical counterparts, and they give rise to new proofs of the Elmendorf theorems for equivariant categories and posets.
\end{abstract}

\tableofcontents

\section{Introduction}

From the standpoint of homotopy theory, there are many different models for topological spaces. Work of Kan \cite{Kan} and others (cf. \cite{GZ}, \cite{MaySimp}, and \cite{QuillenHA}) established that the homotopy theory of spaces is equivalent to the homotopy theory of simplicial sets. Later on, Thomason \cite{ThomasonCat} proved that the homotopy theory of simplicial sets is equivalent to the homotopy theory of small categories, and more recently, Raptis \cite{Raptis} proved that the homotopy theory of small categories is equivalent to the homotopy theory of posets. All told, there are Quillen equivalences
	\[
	\b{Top} \leftrightarrows \b{sSet} \rightleftarrows \b{Cat} \rightleftarrows \b{Pos}
	\]
between the categories of spaces, simplicial sets, small categories, and posets.

The same is true equivariantly. Let $G$ be a discrete group. Then by work of Bohmann-Mazur-Osorno-Ozornova-Ponto-Yarnall \cite{BMOOPY} and May-Stephan-Zakharevich \cite{MSZ} there are Quillen equivalences
	\[
	\b{Top}^G \leftrightarrows \b{sSet}^G \rightleftarrows \b{Cat}^G \rightleftarrows \b{Pos}^G
	\]
between the corresponding categories of left $G$-actions. Thus, $G$-spaces may be modeled by $G$-simplicial sets, small $G$-categories, and $G$-posets, but there are yet more models. Let $\s{O}_G$ denote the orbit category of $G$, i.e. the category of coset spaces $G/H$ and the $G$-equivariant maps between them. Let
	\[
	\b{Fun}(\s{O}_G^\op,\b{Top}) \leftarrow \b{Top}^G : \Phi
	\]
be the functor that sends a $G$-space $X$ to its presheaf of fixed point subspaces $X^H$. Then the functor $\Phi$ has an inverse
	\[
	C : \b{Fun}(\s{O}_G^\op , \b{Top}) \to \b{Top}^G
	\]
up to weak equivalence. Here, a weak equivalence of $G$-spaces is a $G$-map $f : X \to Y$ such that $f^H : X^H \to Y^H$ is a weak homotopy equivalence for all subgroups $H \subset G$, and weak equivalence of presheaves is a natural transformation $\lambda : \s{X} \Rightarrow \s{Y}$ such that each component of $\lambda$ is a weak homotopy equivalence. In this sense, $G$-spaces are also modeled by topological presheaves over $\s{O}_G$.

The original construction of $C$ is due to Elmendorf \cite{Elmendorf}, and thus the existence of homotopy inverse functors $C :  \b{Fun}(\s{O}_G^\op , \b{Top}) \rightleftarrows \b{Top}^G : \Phi$ is typically called Elmendorf's theorem. He defined $C$ using a two-sided bar construction, and Piacenza \cite{Piacenza} later recast things in model categorical terms, with $C$ being the derived left adjoint to $\Phi$. Elmendorf and Piacenza's constructions have distinct advantages. Piacenza's approach generalizes to simplicial, categorical, and posetal contexts, and it produces model-categorical proofs that $G$-simplicial sets, $G$-categories, and $G$-posets can be modeled by presheaves over $\s{O}_G$ (cf. \cite{BMOOPY}, \cite{MSZ}, and \cite{Stephan}). On the other hand, Elmendorf's construction has the virtue of being simple and explicit. It uses no heavy machinery, and it lends itself well to direct analysis (cf. \cite[\S4]{BHOspectra}).

The purpose of this paper is to give similarly direct Elmendorf constructions for $G$-categories and for $G$-posets. We shall define two functors (cf. \S\ref{subsec:catelm} and \S\ref{subsec:poselm})
	\[
	C : \b{Fun}(\s{O}_G^\op,\b{Cat}) \to \b{Cat}^G	\quad\t{and}\quad
	C : \b{Fun}(\s{O}_G^\op,\b{Pos}) \to \b{Pos}^G
	\]
and then prove that they have the expected properties. These functors are elementary. The $C$ functor for $G$-categories is a certain Grothendieck construction. The $C$ functor for $G$-posets is not simply the restriction of this functor, because it generally outputs $G$-preorders. Therefore we must rectify it. Surprisingly, the natural posetal quotient of a $G$-preorder does not usually have the correct equivariant homotopy type, so we use a version of Milnor's infinite join instead. This join construction gives a functorial posetal resolution of every preorder, and it may be of independent interest beyond this paper. 

These constructions stand in sharp contrast to the model categorical analogues of $C$. The latter involve cofibrant resolution in the relevant presheaf categories, and while this is perfectly fine for theoretical considerations, these resolutions are somewhat inexplicit and sometimes difficult to compute. Our constructions have explicit formulas.

We now state our main results. Fix a discrete group $G$, not necessarily finite, and let $B = \abs{\cdot} \circ N : \b{Cat} \to \b{Top}$ denote the classifying space functor.

\begin{thmcomp}Suppose $\s{C} = \b{Cat}$ or $\b{Pos}$. Then for any orbit presheaf $\s{X} : \s{O}_G^\op \to \s{C}$, there is a zig-zag of weak $G$-equivalences between $CB\s{X}$ and $BC\s{X}$.
\end{thmcomp}

In the case $\s{C} = \b{Cat}$, this equivalence may be realized by a single map, and in the case $\s{C} = \b{Pos}$, it may be realized by a zig-zag of two maps. Thus, our Elmendorf constructions are direct lifts of the classical construction to the category of small categories and to the category of posets.

Next, let $\Phi$ denote the fixed point functor generically.

\begin{thmElm} Suppose $\s{C} = \b{Cat}$ or $\b{Pos}$. Then the functors
	\[
	C : \b{Fun}(\s{O}_G^\op,\s{C}) \rightleftarrows \s{C}^G : \Phi
	\]
are inverse up to weak equivalence.
\end{thmElm}

Here, a weak equivalence of categories or posets is a functor $F : \s{D} \to \s{E}$ such that $BF$ is a homotopy equivalence of spaces, a weak equivalence of presheaves is a natural transformation $\lambda : \s{X} \Rightarrow \s{Y}$ such that each component of $\lambda$ is a weak equivalence, and a weak equivalence of $G$-categories or $G$-posets is a $G$-functor $F : \s{D} \to \s{E}$ such that $BF$ is an equivalence of $G$-spaces. Thus, our Elmendorf constructions really are Elmendorf constructions in the usual sense, and they reestablish the equivalence between $\s{C}^G$ and $\b{Fun}(\s{O}_G^\op,\s{C})$ when $\s{C} = \b{Cat}$ or $\b{Pos}$.

We reiterate that the Elmendorf theorems for $G$-categories and $G$-posets were already proven using model-categorical arguments in \cite{BMOOPY} and \cite{MSZ}. The point of this paper is to give concrete formulas for the Elmendorf constructions, and to give new, elementary proofs of the Elmendorf theorems for $G$-categories and $G$-posets. We hope that our techniques will further facilitate the study of equivariant homotopy theory through the lens of $G$-categories and $G$-posets.

\subsection*{Organization} In \S\ref{sec:prelim}, we review some facts about $G$-categories, the Elmendorf theorem, and homotopy colimits. In \S\ref{sec:ElmCat} we introduce the categorical Elmendorf construction and then we prove it has the desired properties. In \S\ref{sec:ElmPos}, we explain how to resolve a $G$-preorder by a $G$-poset, and then we conclude with a discussion of the posetal Elmendorf construction.

\subsection*{Acknowledgements} It is a pleasure to thank Mike Hill, Peter May, and Ang\'{e}lica Osorno for many helpful conversations and for their incisive commentary. This work was partially supported by NSF Grant DMS--1803426.

\section{Preliminaries}\label{sec:prelim}

The goal of this paper is to lift the Elmendorf construction to $G$-categories and to $G$-posets. Our approach is motivated by two observations:
	\begin{enumerate}
		\item{}the Elmendorf construction is a homotopy colimit, and
		\item{}homotopy colimits are modeled by Grothendieck constructions.
	\end{enumerate}
Accordingly, we begin by establishing conventions on $G$-categories, and then we review Elmendorf's theorem and Thomason's homotopy colimit theorem. 

\subsection{$G$-categories} Let $G$ be a discrete group, not necessarily finite. A (small) \emph{$G$-category} is a (small) category $\s{C}$, equipped with a homomorphism $\mu : G \to \b{Aut}(\s{C})$. Thus, for every $g \in G$, we have an invertible functor $g \cdot (-) = \mu(g) : \s{C} \to \s{C}$, and these action maps satisfy the usual associative and unital laws.

A \emph{$G$-functor} $F : \s{C} \to \s{D}$ between $G$-categories is a functor that preserves the action maps, i.e. $[g \cdot (-)] \circ F = F \circ [g \cdot (-)]$ for all $g \in G$. This means $gFx = Fgx$ and $gFf = Fgf$ for every object $x \in \s{C}$, morphism $f \in \s{C}$, and $g \in G$.

A \emph{$G$-natural transformation} between $G$-functors $F,H : \s{C} \rightrightarrows \s{D}$ is a natural transformation $\eta : F \Rightarrow H$ such that $[g \cdot (-)] \circ \eta = \eta \circ [g \cdot (-)]$ as whiskered natural transformations. This means $g\eta_x = \eta_{gx}$ for every $x \in \s{C}$ and $g \in G$.

The collection of all small $G$-categories, $G$-functors, and $G$-natural transformations forms a $2$-category, with componentwise vertical composition and the usual horizontal composition. We write
	\begin{enumerate}[label=(\roman*)]
		\item{}$\b{Cat}^G$ for the $1$-category of all small $G$-categories and $G$-functors,
		\item{}$\underline{\b{Cat}}^G$ for the $2$-category of all small $G$-categories, $G$-functors, and $G$-natural transformations, and 
		\item{}$\b{Fun}^G(\s{C},\s{D})$ for the $1$-category of all $G$-functors and $G$-natural transformations from $\s{C}$ to $\s{D}$.
	\end{enumerate}

We shall frequently consider the fixed point subcategories of a $G$-category $\s{C}$. For any subgroup $H \subset G$, the \emph{$H$-fixed subcategory} $\s{C}^H$ is the subcategory of $\s{C}$ consisting of the $H$-fixed objects and the $H$-fixed morphisms between them. As one might expect, $H$-fixed points are representable. Let $(-)^\t{disc} : \b{Set} \to \b{Cat}$ be the discrete category functor, which sends a set $X$ to the category $X^\t{disc}$, whose object set is $X$, and which has no nonidentity morphisms. Then there is an induced functor $(-)^\t{disc} : \b{Set}^G \to \b{Cat}^G$, and an isomorphism
	\[
	\b{Fun}^G(G/H^\t{disc} , \s{C}) \stackrel{\cong}{\longrightarrow} \s{C}^H ,
	\]
induced by evaluating $G$-functors and $G$-natural transformations at $eH \in G/H^{\t{disc}}$.

The isomorphism above suggests a definition of the fixed point presheaf of $\s{C}$. Since $\underline{\b{Cat}}^G$ is a $2$-category, there is an enriched hom bifunctor
	\[
	\b{Fun}^G(-,-) : (\b{Cat}^G)^\op \times \b{Cat}^G \to \b{Cat},
	\]
which sends a pair of $G$-categories $(\s{C},\s{D})$	to the category $\b{Fun}^G(\s{C},\s{D})$ and a pair of $G$-functors $(F,H)$ to the composition functor $H \circ (-) \circ F$. We thus obtain a composite functor
	\[
	\b{Fun}^G(-,-) \circ [(-)^\t{disc} \times \t{id}] : \s{O}_G^\op \times \b{Cat}^G \to (\b{Cat}^G)^\op \times \b{Cat}^G \to \b{Cat},
	\]
and the transpose
	\[
	\Phi : \b{Cat}^G \to \b{Fun}(\s{O}_G^\op , \b{Cat})
	\]
is the fixed point functor.

For any $G$-category $\s{C}$, the presheaf $\Phi \s{C}$ sends $G/H \in \s{O}_G$ to $\b{Fun}^G(G/H^\t{disc} , \s{C})$ and $f : G/K \to G/H$ to the precomposition map $f^*$. Under the identification $\b{Fun}^G(G/H^\t{disc} , \s{C}) \cong \s{C}^H$, the morphism $r_a : G/K \to G/H$, which maps $eK$ to $aH$, is sent to the multiplication map $\s{C}^K \leftarrow \s{C}^H : a \cdot (-)$. For any $G$-functor $F : \s{C} \to \s{D}$, the corresponding map $\Phi F : \Phi \s{C} \Rightarrow \Phi \s{D}$ is postcomposition by $F$. Under the identification $\b{Fun}^G(G/H^\t{disc} , \s{C}) \cong \s{C}^H$, it corresponds to the restriction of $F$ to $H$-fixed objects and morphisms.

We think of a small $G$-category $\s{C}$ as a model for a $G$-space. To construct a $G$-space from $\s{C}$, we first apply the nerve functor $N : \b{Cat} \to \b{sSet}$ and then we apply the geometric realization functor $\abs{\cdot} : \b{sSet} \to \b{Top}$. We write $B = \abs{\cdot} \circ N$, and refer to $B\s{C}$ as the \emph{classifying space} of $\s{C}$. Homotopical notions for spaces are transported to categories along $B$. In particular, a $G$-functor $F : \s{C} \to \s{D}$ is a weak equivalence if $BF : B\s{C} \to B\s{D}$ is a weak $G$-equivalence, and a weak equivalence $\lambda : \s{X} \Rightarrow \s{Y}$ of presheaves is a natural transformation that is a levelwise weak equivalence. In light of the following proposition, a $G$-functor $F$ is a weak equivalence of $G$-categories if and only if $\Phi F$ is a weak equivalence of presheaves.

\begin{prop}\label{prop:fpsSet} Suppose $G$ is a (possibly infinite) discrete group, $H \subset G$ is a subgroup, and $K$ is a $G$-simplicial set. Then $\abs{K^H}$ is naturally isomorphic to $\abs{K}^H$.
\end{prop}

\begin{cor}\label{cor:fpCat} Suppose $G$ is a (possibly infinite) discrete group, $H \subset G$ is a subgroup, and $\s{C}$ is a $G$-category. Then $B(\s{C}^H)$ is naturally isomorphic to $(B\s{C})^H$.
\end{cor}

Proposition \ref{prop:fpsSet} and its corollary seem to be known, but we could not find proofs in the literature. When $G$ is a finite group, they hold because geometric realization preserves finite limits. It takes just a bit more work in the infinite case. The idea is that the space $\abs{K}^H$ is the subcomplex of $H$-fixed cells in $\abs{K}$, but we give a more categorical argument.

\begin{proof}[Proof of Proposition] The $H$-fixed points $K^H$ fit into an equalizer diagram
	\[
	\begin{tikzpicture}
		\node(KH) at (0,0) {$K^H$};
		\node(K) at (2,0) {$K$};
		\node(P) at (4.5,0) {$\prod_{h \in H} K$};
		\path[->]
		(KH) edge [above] node {$i$} (K)
		(K.15) edge [above] node {$\Delta$} (P.174)
		(K.-15) edge [below] node {$\mu$} (P.186)
		;
	\end{tikzpicture}
	\]
where $\Delta$ is the diagonal and the $h$-component of $\mu$ is the multiplication map $h \cdot (-) : K \to K$. Applying geometric realization yields a diagram
	\[
	\begin{tikzpicture}
		\node(KH) at (0,0) {$\abs{K^H}$};
		\node(K) at (2,0) {$\abs{K}$};
		\node(P) at (4.5,0) {$\abs{\prod_{h \in H}K}$};
		\node(P2) at (4.5,-2) {$\prod_{h \in H} \abs{K}$};
		\path[->]
		(KH) edge [above] node {$\abs{i}$} (K)
		(K.10) edge [above] node {$\abs{\Delta}$} (P.176)
		(K.-10) edge [below] node {$\abs{\mu}$} (P.184)
		(K.-40) edge [above right] node {$\Delta$} (P2.155)
		(K.-69) edge [below left] node {$\lambda$} (P2.163)
		(P) edge [right] node {$\kappa$} (P2)
		;
	\end{tikzpicture}
	\]
where the $h$-component of $\lambda$ is the multiplication map $\abs{h \cdot (-)} : \abs{K} \to \abs{K}$, and $\kappa$ is the canonical map. The top row of this diagram is an equalizer, because realization preserves finite limits. The morphism $\kappa$ is not a homeomorphism when $H$ is infinite, but $\kappa \circ \abs{\Delta} = \Delta$, and $\kappa \circ \abs{\mu} = \lambda$. Therefore $\Delta \circ \abs{i} = \lambda \circ \abs{i}$. It remains to show that $\abs{i}$ is an equalizer of $\Delta$ and $\lambda$.

So suppose $X$ is a space and $f : X \to \abs{K}$ is a continuous map such that $\Delta \circ f = \lambda \circ f$. Choose an element $x \in X$. By the uniqueness of the representation of elements of $\abs{K}$ in non-degenerate form (cf. \cite[\S14]{MaySimp}), we may write $f(x) = \abs{k_n,u_n}$, where $k_n \in K_n$ is non-degenerate, $u_n$ is an interior point of $\Delta^n$, and $\abs{k_n,u_n}$ is the class of $(k_n,u_n)$ in $\abs{K}$. Then $\abs{k_n,u_n} = \abs{hk_n,u_n}$ for every $h \in H$, and it follows that $k_n = hk_n$ for all $h \in H$, because $hk_n$ is also nondegenerate. Consequently,
	\[
	\abs{\Delta}(f(x)) =
	\abs{(k_n)_{h \in H},u_n} = \abs{(hk_n)_{h \in H},u_n}
	= \abs{\mu}(f(x)) ,
	\]
so that $\abs{\Delta} \circ f = \abs{\mu} \circ f$. By the universal property of $\abs{i}$, there is a unique morphism $\widetilde{f} : X \to \abs{K^H}$ such that $\abs{i} \circ \widetilde{f} = f$. This proves $\abs{i}$ is an equalizer of $\Delta$ and $\lambda$.
\end{proof}

We now turn to topology.

\subsection{The topological Elmendorf construction}\label{sec:topElm} Let $G$ be a compact Lie group, and let $J : \s{O}_G \to \b{Top}^G$ be the inclusion. Given any presheaf $\s{X} : \s{O}_G^\op \to \b{Top}$, the topological Elmendorf construction applied to $\s{X}$ is the two-sided bar construction
	\begin{align*}
		C\s{X}	&=	B(\s{X},\s{O}_G,J) .
	\end{align*}
This is the realization of the simplicial space, whose space of $q$-simplices is
	\[
	\coprod_{G/H_0,\dots,G/H_q \in \s{O}_G} \s{X}(G/H_q) \times \s{O}_G(G/H_{q-1},G/H_q) \times \cdots \times \s{O}_G(G/H_0,G/H_1) \times G/H_0 .
	\]
From a conceptual standpoint, the functor $C$ is the derived left adjoint to the fixed point functor $\Phi : \b{Top}^G \to \b{Fun}(\s{O}_G^\op,\b{Top})$, so one might hope that it is homotopy inverse to $\Phi$. This is precisely what Elmendorf proved.

\begin{thm}[\protect{\cite[Theorem 1 and its Corollary]{Elmendorf}}]\label{thm:Elmclass} The functors
	\[
	C : \b{Fun}(\s{O}_G^\op,\b{Top}) \rightleftarrows \b{Top}^G : \Phi
	\]
are inverse up to weak equivalence.
\end{thm}

We could try to lift Elmendorf's construction to $G$-categories and $G$-posets by mimicking the preceding bar construction, but this is ill-advised. The bar construction is a colimit, and colimits of categories are usually intractable. Instead, we shall use Thomason's homotopy colimit theorem. We will review Thomason's theorem in the next section, but for now, we note that Thomason's theorem does not quite apply to the Elmendorf construction $C\s{X}$ as written. Therefore we must repackage the data in $B(\s{X},\s{O}_G,J)$.

When $G$ is a discrete group, the spaces $\s{O}_G(G/H,G/K)$ and $G/H$ are discrete for all subgroups $K,H \subset G$. It follows that
	\[
	B_q(\s{X},\s{O}_G,J) 
	\,\,\,\,
	\cong 
	\!\!\!\!\!\!\!\!
	\coprod_{\tiny{\begin{array}{c}
	G/H_q \stackrel{f_q}{\leftarrow} \cdots \stackrel{f_1}{\leftarrow} G/H_0  \t{ in } \s{O}_G\\
	a_0 H_0 \in G/H_0
	\end{array}}}
	\!\!\!\!\!\!\!\!
	\s{X}(G/H_q) .
	\]
Now we move the elements $a_0 H_0 \in G/H_0$ into the indexing category.

\begin{defn}\label{defn:OG+} Let $\s{O}_{G,+}$ be the category whose objects are pairs $(G/H,aH)$, where $G/H \in \s{O}_G$ is an orbit and $aH \in G/H$ is a coset, and whose morphisms $f : (G/H,aH) \to (G/K , bK)$ are $G$-maps $f : G/H \to G/K$ such that $f(aH) = bK$. The group $G$ acts functorially on $\s{O}_{G,+}$ via
	\[
	g \cdot (G/H,aH) = (G/H,gaH)	\quad\t{and}\quad	g \cdot f = f ,
	\]
and there is a forgetful functor $p : \s{O}_{G,+} \to \s{O}_G$.
\end{defn}

We obtain a further homeomorphism
	\[
	\coprod_{\tiny{\begin{array}{c}
	G/H_q \stackrel{f_q}{\leftarrow} \cdots \stackrel{f_1}{\leftarrow} G/H_0 \t{ in } \s{O}_G, \\
	a_0 H_0 \in G/H_0
	\end{array}}}
	\!\!\!\!\!\!\!\!
	\s{X}(G/H_q) \,\,\,\,\,\,\,\,\,\,\,\, \cong \!\!\!\!\!\!\!\!\!\!\!\!\!\!
	\coprod_{\tiny{
	\begin{array}{c}
	(G/H_q,a_q H_q) \stackrel{f_q}{\leftarrow} \cdots \stackrel{f_1}{\leftarrow} (G/H_0,a_0 H_0) \\
	\t{ in } \s{O}_{G,+}
	\end{array}
	}}
	\!\!\!\!\!\!\!\!\!\!\!\!\!\!\!\!\!\!\!\!\!\!\!\!\!\!
	(\s{X} \circ p)(G/H_q, a_q H_q) ,
	\]
and it becomes a $G$-homeomorphism if we let $G$ act by permuting summands. The right hand side is $B_q(\s{X} \circ p , \s{O}_{G,+},*)$, and we arrive at the following conclusion.

\begin{prop}\label{prop:Elmhocolim} Suppose $G$ is a discrete group. For any $\s{X} : \s{O}_G^\op \to \b{Top}$, there is a $G$-homeomorphism between $C\s{X} = B(\s{X},\s{O}_G,J)$ and $B(\s{X} \circ p,\s{O}_{G,+},*)$.
\end{prop}

The space $B(\s{X} \circ p,\s{O}_{G,+},*)$ is the Bousfield-Kan model of $\t{hocolim}_{\s{O}_{G,+}^\op} (\s{X} \circ p)$  \cite[\S 12.5]{BousfieldKan}, which is precisely what Thomason's theorem concerns.

\subsection{Thomason's homotopy colimit theorem} We now review the Grothendieck construction and Thomason's theorem. Let  $\s{C}$ be a small category, and let $\b{Cat}$ denote the category of small categories. Recall that if $F : \s{C} \to \b{Cat}$ is a functor, then the \emph{Grothendieck construction} on $F$ is the small category $\int_{\s{C}} F$, whose objects are pairs $(c,x)$ such that $c \in \s{C}$ and $x \in Fc$, and whose morphisms $(c,x) \to (d,y)$ are pairs $(f,h)$ such that $f : c \to d$ in $\s{C}$ and $h : Ff(x) \to y$ in $Fd$. Composition is defined by $(f',h') \circ (f,h) = (f' \circ f, h' \circ Ff'(h))$, and $\t{id}_{(c,x)} = (\t{id}_c , \t{id}_x)$.

Thomason's homotopy colimit theorem states that Grothendieck constructions are categorical models for homotopy colimits.

\begin{thm}[\protect{\cite[Theorem 1.2]{Thomasonhocolim}}] Let $F : \s{C} \to \b{Cat}$ be a functor. Then there is a natural homotopy equivalence
	\[
	\eta : \t{hocolim}_{\s{C}} \, NF \stackrel{\simeq}{\longrightarrow} N \Big( \int_{\s{C}} F \Big) ,
	\]
where $N : \b{Cat} \to \b{sSet}$ is the nerve functor.
\end{thm}

We will need a precise description of Thomason's map going forward, and conveniently, Thomason also uses the Bousfield-Kan model of the homotopy colimit. Therefore we understand $\t{hocolim}_{\s{C}} \, NF$ to be the diagonal of the bisimplicial set, whose space of $q$-simplices is
	\[
	\coprod_{
	\tiny{
	\begin{array}{c}
	c_0 \leftarrow \cdots \leftarrow c_q
	\t{ in } \s{C}
	\end{array}
	}
	} NF(c_q) .
	\]
Under this model, the map $\eta : \t{hocolim}_{\s{C}} \, NF \to N(\int F)$ sends a $(q,q)$-simplex
	\[
	(\underbrace{c_0 \stackrel{f_1}{\leftarrow} \dots \stackrel{f_q}{\leftarrow} c_q}_{\t{in } \s{C}} , \underbrace{x_0 \stackrel{h_1}{\leftarrow} \cdots \stackrel{h_q}{\leftarrow} x_q}_{\t{in } F(c_q)})
	\]
to the $q$-simplex
	\[
	\Big( (c_0,F(f_1 \cdots f_q)(x_0) ) \stackrel{k_1}{\leftarrow} (c_1, F(f_2 \cdots f_q)(x_1)) \stackrel{k_2}{\leftarrow} \cdots \stackrel{k_q}{\leftarrow} (c_q , x_q) \Big) ,
	\]
where $k_i = (f_i , F(f_i \cdots f_q)(h_i))$.

We now have all the tools needed to prove the categorical Elmendorf theorem.
	
\section{The construction for $G$-categories}\label{sec:ElmCat}

With Proposition \ref{prop:Elmhocolim} and Thomason's theorem in tow, we now explain how to lift the Elmendorf construction to the category of small $G$-categories. We define the categorical Elmendorf construction, and then we establish its basic properties.

\subsection{The categorical Elmendorf construction}\label{subsec:catelm} Here is the definition.

\begin{defn}\label{defn:catElm} Let $G$ be a discrete group, and suppose $\s{X} : \s{O}_G^\op \to \b{Cat}$ is a presheaf of categories. The \emph{categorical Elmendorf construction} applied to $\s{X}$ is the Grothendieck construction
	\[
	C\s{X} = \int_{\s{O}_{G,+}^\op} (\s{X} \circ p) ,
	\]
where $p : \s{O}_{G,+} \to \s{O}_G$ is the forgetful functor (cf. Definition \ref{defn:OG+}).

We unwind this. For any $\s{X} : \s{O}_G^\op \to \b{Cat}$:
	\begin{enumerate}
		\item{}the objects of $C\s{X}$ are triples $(G/H,aH,x)$ such that $G/H \in \s{O}_G$, $aH \in G/H$, and $x \in \s{X}(G/H)$, and
		\item{}the morphisms of $\int (\s{X} \circ p)$ are pairs $(f,h) : (G/H,aH,x) \to (G/K,bK,y)$ such that $(G/H,aH) \leftarrow (G/K,bK) : f$ in $\s{O}_{G,+}$ and $h : f^* x \to y$ in $\s{X}(G/K)$, where $f^* = \s{X}(f)$.
	\end{enumerate}
Define a $G$-action on $\int (\s{X} \circ p)$ by setting $g \cdot (G/H,aH,x) = (G/H,gaH,x)$ and $g \cdot (f,h) = (f,h)$.

Now let $\lambda : \s{X} \Rightarrow \s{Y}$. Define a $G$-functor $C\lambda : C\s{X} \to C\s{Y}$ by setting 
	\[
	C\lambda(G/H , aH, x) = (G/H, aH, \lambda_{G/H}(x)) \quad\t{and}\quad C\lambda(f,h) = (f,\lambda_{G/K}(h)),
	\]
where $(f,h) : (G/H, aH, x) \to (G/K, bK, y)$. This makes $C$ into a functor.
\end{defn}

By Thomason's theorem, we immediately obtain a \emph{nonequivariant} equivalence
	\[
	\abs{\eta} : C(B\s{X}) = \t{hocolim}_{\s{O}_{G,+}^\op}(B \circ \s{X} \circ p) \stackrel{\simeq}{\longrightarrow}  B \int_{\s{O}_{G,+}^\op} (\s{X} \circ p) = B(C\s{X}) ,
	\]
and we now prove that it respects the $G$-action.

\begin{thm}\label{thm:compElm} Let $G$ be a discrete group. For any orbit presheaf $\s{X} : \s{O}_G^\op \to \b{Cat}$, Thomason's map	
	\[
	\abs{\eta} :  CB\s{X} \to BC\s{X} 
	\]
is a weak $G$-equivalence.
\end{thm}

\begin{proof} We show how to identify the maps that $\eta$ induces on fixed points with other instances of $\eta$. Let $\s{X} : \s{O}_G^\op \to \b{Cat}$ be an orbit presheaf, let $K \subset G$ be a subgroup, and abbreviate $\s{O}_{G,+}^\op$ to $\s{O}$ and $(\s{O}_{G,+}^\op)^K$ to $\s{O}^K$. Then
	\[
	\Big[ \int_{\s{O}} (\s{X} \circ p) \Big]^K = \int_{\s{O}^K} (\s{X} \circ p)\Big|_{\s{O}^K} ,
	\]
because these categories are both the full subcategory of $\int (\s{X} \circ p)$ spanned by the triples $(G/H, aH , x)$ such that $G/H \in \s{O}_G$, $aH \in (G/H)^K$, and $x \in \s{X}(G/H)$. Consequently,
	\[
	\Big[ N \int_{\s{O}} (\s{X} \circ p) \Big]^K = N \Big[ \int_{\s{O}^K} (\s{X} \circ p)\Big|_{\s{O}^K} \Big] .
	\]

Next, we equip $\t{hocolim}_{\s{O}} (N \circ \s{X} \circ p) = \t{diag} \, B_*(N \circ \s{X} \circ p , \s{O}_{G,+} , *) $ with a $G$-action as follows. Let $G$ act on the space
	\begin{align*}
	B_q(N \circ \s{X} \circ p , \s{O}_{G,+} , *)
	&= 
	\!\!\!\!\!\!\!\!
	\coprod_{\tiny{\begin{array}{c}
		(G/H_q , a_q H_q) \leftarrow \cdots \leftarrow (G/H_0,a_0 H_0)	\\
		\t{in } \s{O}_{G,+}
		\end{array}}} 
		\!\!\!\!\!\!\!\!
		(N \circ \s{X})(G/H_q)
	\end{align*}
of $q$-simplices by permuting summands. This simplicial action induces a $G$-action on $\t{diag} \, B_*(N \circ \s{X} \circ p , \s{O}_{G,+} , *)$, whose geometric realization is the $G$-action on the topological Elmendorf construction $CB\s{X}$. Moreover,
	\[
	[\t{diag} \, B_*(N \circ \s{X} \circ p , \s{O}_{G,+} , *) ]^K
	=
	\t{diag} \, B_* ( (N \circ \s{X} \circ p)|_{\s{O}^K} , \s{O}^K , *)
	\]
because on either side, the $q$-simplices are tuples
	\[
	\Big(
	\underbrace{(G/H_0 , a_0 H_0) \to \cdots \to (G/H_q, a_q H_q)}_{\t{in } \s{O}_{G,+}}
	\, , \,
	\underbrace{x_0 \leftarrow \cdots \leftarrow x_q}_{\t{in } \s{X}(G/H_q)}
	\Big)
	\]
such that $a_0 H_0 ,\dots , a_q H_q$ are all $K$-fixed.

Thomason's map $\eta : \t{hocolim} (N \circ \s{X} \circ p) \to N \int (\s{X} \circ p)$ is $G$-equivariant,  and inspecting his formula shows that there is a commutative square
	\[
	\begin{tikzpicture}
		\node(A) at (6,0) {$\Big[ N \int_{\s{O}} (\s{X} \circ p) \Big]^K$};
		\node(B) at (0,0) {$[\t{hocolim}_{\s{O}} (N \circ \s{X} \circ p)]^K$};
		\node(C) at (6,-1.5) {$N\Big[ \int_{\s{O}^K} (\s{X} \circ p)\Big|_{\s{O}^K} \Big]$};
		\node(D) at (0,-1.5) {$\t{hocolim}_{\s{O}^K} (N \circ \s{X} \circ p)|_{\s{O}^K}$};
		\path[->]
		(A) edge [right] node {$=$} (C)
		(B) edge [left] node {$=$} (D)
		(B) edge [above] node {$\eta^K$} (A)
		(D) edge [below] node {$\eta$} (C)
		;
	\end{tikzpicture} .
	\]
The bottom map is an equivalence by Thomason's homotopy colimit theorem, which implies that $\eta^K$ is an equivalence for all subgroups $K \subset G$. Applying the geometric realization functor $\abs{\cdot} : \b{sSet} \to \b{Top}$ and using Proposition \ref{prop:fpsSet} to commute $(-)^K$ with $\abs{\cdot}$ shows that
	\[
	\abs{\eta}^K : (CB\s{X})^K \to (BC\s{X})^K
	\]
is a homotopy equivalence for all $K$. Hence $\abs{\eta}$ is a weak $G$-equivalence.
\end{proof}

\subsection{Elmendorf's theorem for $G$-categories} 

One would hope that the categorical Elemendorf construction $C : \b{Fun}(\s{O}_G^\op,\b{Cat}) \to \b{Cat}^G$ is homotopy inverse to taking fixed points, but this is not automatic from Theorem \ref{thm:compElm}. Therefore we shall prove it.

\begin{thm}\label{thm:catcoal} For any discrete group $G$, the functors
	\[
	C : \b{Fun}(\s{O}_G^\op,\b{Cat}) \rightleftarrows \b{Cat}^G : \Phi
	\]
are inverse up to natural weak equivalence.
\end{thm}

The required transformations $C \circ \Phi \Rightarrow \t{id}$ and $\Phi \circ C \Rightarrow \t{id}$ are constructed in Propositions \ref{prop:wePhiC} and \ref{prop:weCPhi}.

\begin{prop}\label{prop:wePhiC} There is a natural weak equivalence $\ve : \Phi \circ C \Rightarrow \t{id}$.
\end{prop}

\begin{proof} Let $\s{X} : \s{O}_G^\op \to \b{Cat}$ be an orbit presheaf, and $L \subset G$ be a subgroup. The $L$-fixed subcategory of $(C\s{X})^L$ is the full subcategory of $C\s{X}$ spanned by the triples $(G/H,aH,x)$ such that $aH \in (G/H)^L$. Construct a functor
	\[
	(\ve_\s{X})_L : (C\s{X})^L \to \s{X}(G/L)
	\]
as follows:
	\begin{enumerate}
		\item{}Given any $L$-fixed object $(G/H,aH,x)$, let $r_a : G/L \to G/H$ be the $G$-map sending $eL$ to $aH$, and let $\s{X}(G/L) \leftarrow \s{X}(G/H) : r_a^*$ be its image under $\s{X}$. We define
			\[
			(\ve_{\s{X}})_L(G/H,aH,x) \stackrel{\t{def}}{=} r_a^* x.
			\]
		\item{}Given any morphism $(f,h) : (G/H,aH,x) \to (G/K,bK,y)$ between $L$-fixed triples, let $r_a : G/L \to G/H$ and $r_b : G/L \to G/K$ be as above, and note that $f \circ r_b = r_a$. We define
			\[
			(\ve_{\s{X}})_L(f,h) \stackrel{\t{def}}{=} r_b^* h : r_a^* x = r_b^* f^* x \to r_b^* y.
			\]
	\end{enumerate}
Straightforward checks show that $(\ve_{\s{X}})_L$ is natural in $L$ and $\s{X}$, i.e. $\ve$ is a natural transformation $\ve : \Phi \circ C \Rightarrow \t{id}_{\b{Fun}(\s{O}_G^\op,\b{Cat})}$.

It remains to check that $\ve_{\s{X}} : \Phi C \s{X} \Rightarrow \s{X}$ is a weak equivalence for every $\s{X}$. We shall construct a homotopy inverse to $\ve_L = (\ve_{\s{X}})_L : (C\s{X})^L \to \s{X}(G/L)$ for every $G/L$. Given any object $x \in \s{X}(G/L)$, let $\eta_L(x) = (G/L,eL,x) \in (C\s{X})^L$, and given any morphism $f : x \to y$ in $\s{X}(G/L)$, let $\eta_L(f) = (\t{id}_{G/L} , f)$. Then $\eta_L$ is a functor, and $\ve_L \circ \eta_L = \t{id}$. The composite $\eta_L \circ \ve_L$ is not the identity, but
	\[
	( r_a, \t{id}_{r_a^* x} ) : (G/H,aH,x) \to (G/L , eL , r_a^* x) = \eta_L \ve_L(G/H,aH,x)
	\]
defines a natural transformation $\t{id} \Rightarrow \eta_L \circ \ve_L$. Therefore the classifying space functor $B : \b{Cat} \to \b{Top}$ sends $\ve_L$ and $\eta_L$ to homotopy inverse functions. It follows that $B\ve_L$ is a homotopy equivalence for all $L$, and that $\ve : \Phi C \s{X} \Rightarrow \s{X}$ is a levelwise weak equivalence in $\b{Fun}(\s{O}_G^\op , \b{Cat})$.
\end{proof}

\begin{rem} As with the topological Elemendorf construction, the map $\eta_L$ does not vary naturally in $L$. Nevertheless, $\ve$ is a natural weak equivalence.
\end{rem}

\begin{prop}\label{prop:weCPhi} There is a natural weak equivalence $\t{ev} : C \circ \Phi \Rightarrow \t{id}$. 
\end{prop}

\begin{proof} Suppose $\s{C}$ is a $G$-category, and consider $C \Phi \s{C}$. If $(G/H,aH,x)$ is an object of $C \Phi \s{C}$, then $x \in \s{C}^H$, and if $(f,h) : (G/H,aH,x) \to (G/K,bK,y)$ is a morphism of $C \Phi \s{C}$, then $h \in \s{C}^K$. Define a $G$-functor $\t{ev}_{\s{C}} : C \Phi \s{C} \to \s{C}$ by
	\[
	\t{ev}_{\s{C}}(G/H,aH,x) = ax	\quad\t{and}\quad	\t{ev}_{\s{C}}\Big( (f,h) : (G/H,aH,x) \to (G/K,bK,y) \Big) = bh.
	\]
One readily checks that this defines a natural transformation $\t{ev} : C \circ \Phi \Rightarrow \t{id}_{\b{Cat}^G}$, and that $\Phi \t{ev}_{\s{C}} = \ve_{\Phi \s{C}}$. The natural transformation $\ve_{\Phi \s{C}}$ is a weak equivalence by Proposition \ref{prop:wePhiC}, and it follows that $\t{ev}_{\s{C}}$ is also a weak equivalence, because $\Phi$ creates such morphisms. Therefore $\t{ev}$ is a natural weak equivalence.
\end{proof}

\subsection{Example: equivariant universal spaces}\label{app:univ} With the basic properties of the categorical Elmendorf construction established, we now consider an example. Equivariant universal spaces arise from the simplest orbit presheaves, and testing on them reveals a fundamental issue with the posetal Elmendorf construction.

 Let $\s{F}$ be a family of subgroups of $G$, and let $\s{X}_{\s{F}} : \s{O}_G^\op \to \b{Cat}$ be defined by
	\[
	\s{X}_{\s{F}}(G/H) = \left\{
	\begin{array}{cl}
		*	&\t{if } H \in \s{F}	\\
		\varnothing	&\t{if } H \notin \s{F}
	\end{array}
	\right. ,
	\]
where $*$ is the terminal category and $\varnothing$ is the initial one. The closure of $\s{F}$ under subconjugacy ensures that $\s{X}_{\s{F}}$ is a well-defined presheaf of categories.

Now regard a preorder $P$ as a category with at most one morphism $x \to y$ for any $x,y \in P$. Given such a category $\s{P}$, we recover a preorder on $\t{Ob}(\s{P})$ by declaring $x \leq y$ if there is a morphism $x \to y$. Then $\s{X}_{\s{F}}$ is actually a presheaf of posets, but this does not quite imply that $C\s{X}$ is a poset. Indeed, applying the Elmendorf construction to $\s{X}_{\s{F}}$ yields a \emph{$G$-preorder} $C\s{X}_{\s{F}}$, whose underlying $G$-set is
	\[
	C\s{X}_{\s{F}} = \coprod_{H \in \s{F}} G/H ,
	\]
and whose order relation is
	\[
	(H,aH) \leq (K,bK)	\quad\t{if and only if}\quad	aHa^{-1} \supset bKb^{-1} .
	\]
Specializing to the case $\s{F} = \{e\}$, we obtain
	\[
	C\s{X}_{\{e\}} = G/e
	\]
with $x \leq y$ for all $x,y \in G/e$. This is the standard categorical model of $EG$.

On the other hand, the categorical Elmendorf construction does send presheaves of preorders to $G$-preorders.

\begin{lem}\label{lem:Cpreord} If $\s{X} : \s{O}_G^\op \to \b{Cat}$ is valued in preorders, then $C\s{X}$ is a $G$-preorder.
\end{lem}

The proof is straightforward. This lemma is the starting point for our posetal Elmendorf construction.

\section{The construction for $G$-posets}\label{sec:ElmPos}

In this section, we explain how to lift Elmendorf's construction to the category of $G$-posets. The natural guess is to restrict the categorical Elmendorf construction $C\s{X} \cong \int (\s{X} \circ p)$ to presheaves of posets. The result is guaranteed to be a $G$-preorder by Lemma \ref{lem:Cpreord}, but it need not be a $G$-poset, as demonstrated by the presheaves $\s{X}_{\s{F}}$ in \S\ref{app:univ}. Therefore we must devise a method for converting $G$-preorders into $G$-posets, without changing their equivariant homotopy type.

\subsection{Posetal quotients}\label{sec:posquot} There is an adjunction $\ol{(-)} : \b{Preord} \rightleftarrows \b{Pos} : i$ between the category of preorders and the category of posets. The right adjoint $i$ is the inclusion functor, and the left adjoint $\ol{(-)}$ sends a preorder to its posetal quotient. This is the poset obtained by taking the quotient of a preorder $(P,\leq)$ by the equivalence relation
	\[
	x \sim y	\quad\t{if and only if}\quad		x \leq y \t{ and } y \leq x ,
	\]
and then declaring $[x] \leq [y]$ if $x \leq y$ for some choice of representatives.

For any preorder $P$, there is a quotient map $\pi : P \to \ol{P}$, and choosing a set of representatives for $\sim$ defines a section $P \leftarrow \ol{P} : s$. Thus $\pi \circ s = \t{id}$. The composite $s \circ \pi$ is not the identity, but for any $x \in P$, we have $x \leq s[x] \leq x$, i.e. there are natural transformations $\t{id}_P \Rightarrow s \circ \pi \Rightarrow \t{id}_P$. On passage to classifying spaces, we obtain a homotopy equivalence $BP \simeq B\ol{P}$.

Unfortunately, the equivalence $BP \simeq B\ol{P}$ does not hold equivariantly.

\begin{cex} Consider the $G$-preorder $P = C \s{X}_{\{e\}}$ from \S\ref{app:univ}. Then $\ol{P} = *$, and applying the classifying space functor gives $BP \simeq EG$ and $B\ol{P} = *$.
\end{cex}

We need another construction. In keeping with the rest of homotopy theory, we avoid taking a quotient and instead construct a resolution.

\subsection{Posetal resolutions}\label{subsec:Milnor} We now explain how to resolve a preorder by a poset, in such a way that preserves equivariant homotopy types.

By way of motivation, consider the space $EG$ once more. Then
	\[
	EG \simeq B(*,G,G) = B C \s{X}_{\{e\}},
	\]
but the space $EG$ can also be modeled by the Milnor join
	\[
	EG \simeq \t{colim}\, \underbrace{G * G * \cdots * G}_{n \t{ times}} .
	\]
When $G = \{g_1,\dots,g_n \}$ is a finite group, the Milnor join is the classifying space of the countably infinite poset
	
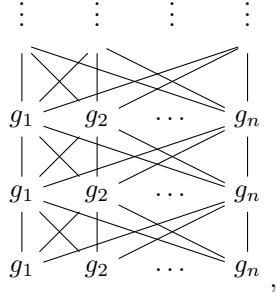
\[
	\begin{tikzpicture}
		\node(10) at (1,0) {$g_1$};
		\node(20) at (2,0) {$g_2$};
		\node(30) at (3,0) {$\dots$};
		\node(40) at (4,0) {$g_n$};
		\node(11) at (1,1) {$g_1$};
		\node(21) at (2,1) {$g_2$};
		\node(31) at (3,1) {$\dots$};
		\node(41) at (4,1) {$g_n$};
		\node(12) at (1,2) {$g_1$};
		\node(22) at (2,2) {$g_2$};
		\node(32) at (3,2) {$\dots$};
		\node(42) at (4,2) {$g_n$};
		\node(13) at (1,3) {};
		\node(23) at (2,3) {};
		\node(33) at (3,3) {};
		\node(43) at (4,3) {};
		\node(14) at (1,3.5) {$\vdots$};
		\node(24) at (2,3.5) {$\vdots$};
		\node(34) at (3,3.5) {$\vdots$};
		\node(44) at (4,3.5) {$\vdots$};
		\path[-]
		(10)	edge node {} (11)
			edge node {} (21)
			edge node {} (41)
		(20)	edge node {} (11)
			edge node {} (21)
			edge node {} (41)
		(40)	edge node {} (11)
			edge node {} (21)
			edge node {} (41)
			
		(11)	edge node {} (12)
			edge node {} (22)
			edge node {} (42)
		(21)	edge node {} (12)
			edge node {} (22)
			edge node {} (42)
		(41)	edge node {} (12)
			edge node {} (22)
			edge node {} (42)
			
		(12)	edge node {} (13)
			edge node {} (23)
			edge node {} (43)
		(22)	edge node {} (13)
			edge node {} (23)
			edge node {} (43)
		(42)	edge node {} (13)
			edge node {} (23)
			edge node {} (43)
		;
	\end{tikzpicture} ,
	\]
where each vertex $g_i$ represents an element of the poset, and each upward edge represents an order relation $\leq$. The classifying space of the first row is $G$, the classifying space of the first two rows is the join $G*G$, and so on.

Formally, the poset above is the set $MG = G \times \bb{N}$ equipped with the relation
	\[
	(x,m) < (y,n)	\quad\t{if and only if}\quad		m<n .
	\]
The classifying space $BMG$ is nonequivariantly contractible because $MG$ is a filtered poset (cf. \cite[\S1]{QuillenHA}), and the group $G$ acts freely on $BMG$ because it acts freely on $MG$. Therefore $BMG \simeq EG$. The first coordinate projection $\pi : MG \to C\s{X}_{\{e\}}$ is order-preserving, and it induces an equivalence on classifying spaces. Therefore we regard $MG$ as a posetal resolution of $C\s{X}_{\{e\}}$.

We now generalize this $M$ construction to arbitrary $G$-preorders.

\begin{defn}\label{defn:Milnor} Suppose $(P,\leq)$ is a $G$-preorder. We define $MP$ to be the set
	\[
	MP = P \times \bb{N}
	\]
equipped with the partial order
	\begin{center}
		$(x,m) < (y,n)$ if and only if $x \leq_P y$ and $m < n$.
	\end{center}
We define a $G$-action on $MP$ by $g \cdot (x,m) = (gx,m)$ for all $g \in G$ and $(x,m) \in MP$. This makes $MP$ into a $G$-poset.

Next, given any map $f : P \to Q$ of $G$-preorders, we define
	\[
	Mf = f \times \t{id} : MP \to MQ.
	\]
This makes $M$ into a functor $M : \b{Preord}^G \to \b{Pos}^G$.
\end{defn}

The first coordinate projection $\pi : MP \to P$ is a natural order-preserving $G$-map, and we now prove that is a resolution of $P$. We begin with the nonequivariant result. In the following argument, we understand a \emph{subposet} of a poset $Q$ to be a subset $S \subset Q$ equipped with the restriction of $Q$'s order relation, i.e. if $x,y \in S$, then $x \leq_S y$ if and only if $x \leq_{Q} y$.

\begin{prop}\label{prop:Mpieq} For any nonequivariant preorder $P$, the projection $\pi : MP \to P$ induces a weak equivalence $B\pi : BMP \to BP$.
\end{prop}

\begin{proof} We compute the comma categories of $\pi : MP \to P$ and then apply Quillen's Theorem A. Consider the comma category $x \downarrow \pi$ for some $x \in P$. It can be identified with the subposet $\{ (y,n) \in MP \, | \, x \leq y \t{ in } P\} \subset MP$, and the latter contains $\pi^{-1}[x] = \{(y,n) \in MP \, | \, x \leq y \leq x \t{ in } P\}$ as a subposet. Here $[x]$ denotes the equivalence class of $x$ under the relation in \S\ref{sec:posquot}. We claim that the inclusion $i : \pi^{-1}[x] \hookrightarrow x \downarrow \pi$ is a homotopy equivalence.

Define a monotone map $r : x \downarrow \pi \to \pi^{-1}[x]$ by
	\[
	r(y,n) = \left\{
		\begin{array}{cl}
			(y,n)	&	\t{if } y \in [x] \\
			(x,n)	&	\t{if } y \notin [x]
		\end{array}
		\right. .
	\]
Then $r \circ i = \t{id} : \pi^{-1}[x] \to \pi^{-1}[x]$ and $Br \circ Bi$ is the identity on $B(\pi^{-1}[x])$. The composite $i \circ r$ is not the identity, but it is connected to the identity through a zig-zag of inequalities. Let $s : x \downarrow \pi \to x \downarrow \pi$ be the monotone map defined by
	\[
	s(y,n) = \left\{
		\begin{array}{cl}
			(y,n)	&	\t{if } y \in [x] \\
			(y,n+1)	&	\t{if } y \notin [x]
		\end{array}
		\right. .
	\]
Then $(y,n) \leq s(y,n) \geq ir(y,n)$ for all $(y,n) \in x \downarrow \pi$. Therefore $Bi \circ Br$ is homotopic to the identity on $B(x \downarrow \pi)$, and therefore $Bi : B(\pi^{-1}[x]) \to B(x \downarrow \pi)$ is a homotopy equivalence.

Now for any $(y,m) , (z,n) \in \pi^{-1}[x]$, we have the inequalities $y \leq x \leq z$, and therefore $(y,m) < (z,n)$ if and only if $m < n$. It follows $\pi^{-1}[x]$ is filtered, because any two elements $(y,m)$ and $(z,n)$ are bounded above by $(x,\max(m,n)+1)$. Therefore $B(x \downarrow \pi) \simeq B(\pi^{-1}[x]) \simeq *$, and Quillen's Theorem A implies that $B\pi : BMP \to BP$ is a homotopy equivalence.
\end{proof}

Now we boost Proposition \ref{prop:Mpieq} up to an equivariant result. Unlike the posetal quotient $\pi : P \to \ol{P}$, the projection map $\pi : MP \to P$ plays well with the $G$-action.

\begin{thm}\label{thm:MGeq} Let $G$ be a discrete group. For any $G$-preorder $P$, the natural projection $\pi : MP \to P$ is a weak equivalence of $G$-preorders.
\end{thm}

\begin{proof} For any subgroup $H \subset G$, there is an equality $(MP)^H = M(P^H)$ of posets, and a commutative square
	\[
	\begin{tikzpicture}
		\node(A) at (0,0) {$(MP)^H$};
		\node(B) at (3,0) {$P^H$};
		\node(C) at (0,-1.5) {$M(P^H)$};
		\node(D) at (3,-1.5) {$P^H$};
		\path[->]
		(A) edge [above] node {$\pi^H$} (B)
		(C) edge [below] node {$\pi$} (D)
		(A) edge [left] node {$=$} (C)
		(B) edge [right] node {$=$} (D)
		;
	\end{tikzpicture} .
	\]
By Proposition \ref{prop:Mpieq}, the bottom map is an equivalence on classifying spaces, and therefore the top map is, too. Thus $\pi^H$ is an equivalence for all subgroups $H \subset G$, i.e. $\Phi \pi$ is an equivalence, and therefore $\pi$ is also an equivalence because $\Phi$ creates these morphisms.\end{proof}

\begin{cor}The functors $M : \b{Preord}^G \rightleftarrows \b{Pos}^G : \t{inc}$ are inverse up to natural weak equivalence.
\end{cor}

\subsection{The Elmendorf theorem for $G$-posets}\label{subsec:poselm}

We conclude by discussing Elmendorf's theorem for $G$-posets. For clarity, we shall write $C_{cat}$ for the categorical Elmendorf construction and $C_{pos}$ for the posetal one.

\begin{defn} Let $G$ be a discrete group, and suppose $\s{X} : \s{O}_G^\op \to \b{Pos}$ is a presheaf of posets. The \emph{posetal Elmendorf construction} applied to $\s{X}$ is
	\[
	C_{pos}\s{X} = M C_{cat} \s{X},
	\]
where $C_{cat}$ and $M$ are the functors from Definitions \ref{defn:catElm} and \ref{defn:Milnor}.
\end{defn}

Unwinding everything shows that the underlying $G$-set of $C_{pos}\s{X}$ is
	\[
	\coprod_{G/H \in \s{O}_G} G/H \times \s{X}(G/H) \times \bb{N},
	\]
with $G$ acting on the $G/H$ factors only. Its elements are quadruples $(G/H,aH,x,m)$, where $G/H \in \s{O}_G$, $aH \in G/H$, $x \in \s{X}(G/H)$, and $m \in \bb{N}$. The partial order on $C_{pos}\s{X}$ is defined by
	\[
	(G/H,aH,x,m) < (G/K,bK,y,n)	\quad\t{if and only if}\quad
	\left\{
		\begin{array}{l}
			aHa^{-1} \supset bKb^{-1},	\\
			r_{b^{-1}a}^*x \leq y \t{ in } \s{X}(G/K),	\\
			\t{and } m < n
		\end{array}
	\right. 
	\]
where $r_{b^{-1}a} : G/K \to G/H$ sends $eK$ to $b^{-1}aH$.

As with the categorical Elmendorf construction, the posetal Elmendorf construction is compatible with the original.

\begin{thm}\label{thm:Elmpostop} Let $G$ be a discrete group. For any orbit presheaf $\s{X} : \s{O}_G^\op \to \b{Pos}$, there is a zig-zag of natural weak $G$-equivalences between $CB\s{X}$ and $BC_{pos} \s{X}$.
\end{thm}

\begin{proof}The maps
	\[
	CB\s{X} \stackrel{\abs{\eta}}{\longrightarrow} B \Big( \int_{\s{O}_{G,+}} (\s{X} \circ p) \Big) = BC_{cat}\s{X} \stackrel{B\pi}{\longleftarrow} BC_{pos}\s{X}
	\]
are both weak $G$-equivalences by Theorems \ref{thm:compElm} and \ref{thm:MGeq}.
\end{proof}

Moreover, the posetal Elmendorf theorem holds for $C_{pos}$.

\begin{thm}\label{thm:Elmpos} For any discrete group $G$, the functors
	\[
	C_{pos} : \b{Fun}(\s{O}_G^\op , \b{Pos} ) \rightleftarrows \b{Pos}^G : \Phi
	\]
are inverse up to natural weak equivalence.
\end{thm}

\begin{proof}There are equivalences
	\[
	\Phi C_{pos} \s{X} = \Phi MC_{cat} \s{X} \stackrel{\Phi \pi}{\longrightarrow} \Phi C_{cat} \s{X} \stackrel{\ve}{\longrightarrow} \s{X}
	\]
by Theorem \ref{thm:MGeq} and Proposition \ref{prop:wePhiC}, and there are equivalences 
	\[
	C_{pos}\Phi\s{X} = M C_{cat} \Phi\s{X} \stackrel{\pi}{\longrightarrow} C_{cat} \Phi\s{X} \stackrel{\t{ev}}{\longrightarrow} \s{X}
	\]
by Theorem \ref{thm:MGeq} and Proposition \ref{prop:weCPhi}.
\end{proof}

\end{document}